\documentclass[12pt]{article}

\usepackage[reqno,tbtags]{amsmath}
\usepackage[latin1]{inputenc}
\usepackage[english]{babel}
\usepackage[colorlinks]{hyperref}
\usepackage{mathtools}
\usepackage{amssymb}
\usepackage{amsthm}
\usepackage{mathrsfs}
\usepackage{enumerate}

\DeclarePairedDelimiter{\abs}{\lvert}{\rvert}
\DeclarePairedDelimiter{\norm}{\lVert}{\rVert}
\DeclarePairedDelimiter{\inner}{\langle}{\rangle}

\newcommand{\inv}{^{-1}}
\newcommand{\trans}{\mathcal{L}}
\newcommand{\pret}{\mathscr{L}}
\newcommand{\cros}{C(X)\rtimes_{\alpha,\trans}\N}
\newcommand{\N}{\mathbb{N}}
\newcommand{\Z}{\mathbb{Z}}
\newcommand{\T}{\mathbb{T}}
\newcommand{\cs}{C^*}

\newcommand{\hilbert}{\mathsf{H}}

\DeclareMathOperator{\supp}{supp}
\DeclareMathOperator{\ind}{ind}

\newtheorem{theorem}{Theorem}
\newtheorem{lemma}[theorem]{Lemma}
\newtheorem{proposition}[theorem]{Proposition}

\title{{\bf On the Exel crossed product
of topological covering maps} }

\author{{\small {\sc Toke Meier Carlsen}
\footnote{Present address: Department of Mathematical Sciences,
The Norwegian University of Science and Technology (NTNU),
NO-7491 Trondheim, Norway. \newline e-mail: tokemeie@math.ntnu.no}} \\
{\small {\it Department of Mathematics \& Computer Science,}}\\
{\small {\it  University of Southern Denmark,}}\\
{\small {\it Campusvej 55, DK-5230 Odense M, Denmark}} \\*[0.5cm]
{\small {\sc Sergei Silvestrov}} \\
{\small {\it Centre for Mathematical Sciences, Lund university,}}\\
{\small {\it Box 118, SE-22100 Lund, Sweden.} }\\
{\small {\it e-mail: sergei.silvestrov@math.lth.se } } }

\date{January 10, 2008}

\begin{document}

\maketitle

%\vspace*{2.5cm}
% \pagebreak

\newpage

\begin{abstract}
\noindent For dynamical systems defined by a
covering map of a compact Hausdorff space and the
corresponding transfer operator, the associated
crossed product $C^*$-algebras $\cros$
introduced by Exel and Vershik are considered. An important
property for homeomorphism dynamical systems is
topological freeness. It can be extended in a
natural way to in general non-invertible
dynamical systems generated by covering maps. In
this article, it is shown that the following four
properties are equivalent: the dynamical system
generated by a covering map is topologically
free; the canonical imbedding of $C(X)$ into
$\cros$ is a maximal abelian $C^*$-subalgebra of
$\cros$; any nontrivial two sided ideal of
$\cros$ has non-zero intersection with the
imbedded copy of $C(X)$; a certain natural
representation of $\cros$ is faithful. This
result is a generalization to non-invertible
dynamics of the corresponding results for crossed
product $C^*$-algebras of homeomorphism dynamical
systems.
\end{abstract}

\footnote{ 
\\
 \noindent
 This work was supported by the Swedish
Foundation for International Cooperation in
Research and Higher Education (STINT), the Swedish Research Council, the
Crafoord Foundation, the Royal Swedish Academy of Sciences and
the Royal Physiographic Society in Lund. The first named author
was also supported by the Danish Natural Science Research Council.}

\section{Introduction}
A dynamical system generated by a homeomorphism
of a compact Hausdorff topological space leads to
crossed product $C^*$-algebras of continuous
functions on the space by the action of the
additive group of integers via composition of
continuous functions with the iterations of
the generating homeomorphism. The interplay between
topological properties of the dynamical system
(or more general actions of groups) such as
minimality, transitivity, freeness and others on
the one hand, and properties of ideals,
subalgebras and representations of the
corresponding crossed product $C^*$-algebra on
the other hand have been a subject of intensive
investigations at least since the 1960's. This
interplay and its implications for operator
representations of the corresponding crossed
product algebras, spectral and harmonic analysis
and non-commutative analysis and non-commutative
geometry are fundamental for the mathematical
foundations of quantum mechanics, quantum field
theory, string theory, integrable systems,
lattice models, quantization, symmetry analysis
and, as it has become
clear recently, in wavelet analysis and its
applications in signal and image processing (see
\cite{BratJorgbook,Davidson,Jorg-b-1,
MACbook1,MACbook2,MACbook3,
OstSam-book,Ped-book,Williamsbook} and references
therein).

In the works of Zeller-Meier \cite{Zeller-Meier},
Effros, Hahn \cite{EffrosHahn}, Elliott
\cite{Elliott}, Archbold, Quigg, Spielberg
\cite{ArchbordSpielberg,QuiggSpielberg,Spielberg},
Kishimoto, Kawamura, Tomiyama
\cite{KawamuraTomiama,Kishimoto,Tomiama:SeoulLN}
it was observed that the property of topological
freeness of the dynamics for a homeomorphism, or
for more general actions of groups, is closely
linked with the structure of the ideals in the
corresponding crossed product $C^*$-algebra and
in particular with the existence of non-zero
intersections between ideals and the algebra of
continuous functions imbedded as a $C^*$-subalgebra
into the crossed product $C^*$-algebra. The
property of topological freeness has also
been observed to be equivalent or closely
linked to the position of the algebra of
continuous functions inside the crossed product,
namely with wether it is a maximal abelian
subalgebra or not.  (For recent developments in
this direction for reversible dynamical systems
see also \cite{SSD1,SSD2,SSD3,SvT}.) This
interplay has been considered both for the
universal crossed product $C^*$-algebra and for
the reduced crossed product $C^*$-algebra, the
later providing one of the important insights
into the significance of those properties for
representations of the crossed product.

One of the basic motivating points for this paper
is the following pi\-vo\-tal theorem summarizing
results established in
\cite{ArchbordSpielberg,Elliott,KawamuraTomiama,
Kishimoto,Tomiama:SeoulLN,Zeller-Meier}. This
theorem concerns the dynamical systems generated
by a homeomorphism of a compact
Hausdorff topological space, and
establishes for such dynamical systems
equi\-valences between topological freeness and
properties of ideals and of the canonical
subalgebra of continuous functions inside the
crossed product $C^*$-algebra. It can be found
for example in the book by Tomiyama \cite[Theorem
5.4]{{Tomiama:SeoulLN}} in the following
convenient formulation.

\begin{theorem} \label{Thm:TomBequivThm54} The following three properties are
equivalent for a compact Hausdorff space $X$ and
a homeomorphism $\sigma$ of $X$:
\begin{itemize}
\item[\rm (i)] The non-periodic points of $(X, \sigma)$ are dense in $X$;
\item[\rm (ii)] Any non-zero closed ideal $I$
in the crossed product
 $C^*$-algebra \\
 $C(X) \rtimes_{\alpha} \mathbb{Z}$ satisfies $I \cap
 C(X) \neq
 \{0\}$;
\item[\rm (iii)] $C(X)$ is a maximal abelian $C^*$-subalgebra of $C(X)
\rtimes_{\alpha} \mathbb{Z}$.
\end{itemize}
\end{theorem}

In \cite{SSD1,SSD2,SSD3,SvT}, extensions and
modifications of this result and the interplay
between dynamics and properties of the canonical
subalgebra and ideals have been investigated for
dynamical systems that are not topolo\-gically free
on more general spaces than Hausdorff topological
spaces, both in the context of algebraic crossed
product by $\mathbb{Z}$ and for the corresponding
Banach algebra and $C^*$-algebra crossed products
in the case of homeomorphism dynamical systems or
more general dynamical systems generated by
invertible map. Also in these works, this
interplay has been considered on the side of
representations as well as with respect to
duality in the crossed product algebras.

In the recent years, substantial efforts are
made in establishing broad interplay
between $C^*$-algebras and non-invertible
dynamical systems, actions of semigroups,
equivalence relations, (semi-)groupoids,
correspondences (see for example works by Exel,
Arzumanian, Vershik
\cite{ExelMR1966826,ExelMR2032486,ExelVershikMR2195591,ArumzVershik1,ArumzVershik2},
Deaconu \cite{DeaconuMR1233967}, Renault
\cite{Renault1,Renault2MR1770333,Renault3}, Adji, an Huef,
Laca, Nielsen, Raeburn
\cite{AdjiLacaNielsenRaeburn,anHuefRaeburnMR98k:46098},
Bratteli, Dutkay, Jorgensen, Evans
\cite{BratEvansJorg,BratJorgIFSAMSmemo,BratJorgbook,DutkayJorg1,DutkayJorg2,DutkayJorg3,DutkayJorg4,JorgWavSignFracbook},
Dai, Larson \cite{DaiLarsonMeAMS}, Kawamura,
Kajiwara, Watatani
\cite{KajiwaraWatataniComlDyn,Kawamura1,Watatani1},
Ostrovsky\u{\i}, Samo\u{\i}lenko
\cite{OstSam-book} Cuntz, Krieger
\cite{CuntzKr1}, Matsumoto \cite{Matsumoto1},
Eilers \cite{Eilers1}, Carlsen, Silvestrov
\cite{CarlsenSilv} and references therein).

The notion of topological freeness for dynamical systems generated by
a homeomorphism can in a natural way be extended to such in general
non-invertible dynamical systems. The main result of this paper is
\mbox{Theorem \ref{th:equivCovMaps},} extending Theorem  \ref{Thm:TomBequivThm54} to non-invertible dynamical systems
generated by covering maps on compact Hausdorff spaces and to the
corresponding crossed product $C^*$-algebras
$C(X)\rtimes_{\alpha,\mathcal{L}}\mathbb{N}$ introduced by Exel and
Vershik in \cite{ExelVershikMR2195591}.
We also add a fourth equivalent condition of faithfulness
of a certain specified explicitly representation
of $C(X)\rtimes_{\alpha,\mathcal{L}}\mathbb{N}$.
More precisely, in
Theorem~\ref{th:equivCovMaps}, we show that the following four properties are equivalent:
the dynamical system generated by a covering map is topologically
free; the canonical imbedding of $C(X)$ into
$C(X)\rtimes_{\alpha,\mathcal{L}}\mathbb{N}$ is a maximal abelian
$C^*$-subalgebra of $C(X)\rtimes_{\alpha,\mathcal{L}}\mathbb{N}$; any
nontrivial two sided ideal of
$C(X)\rtimes_{\alpha,\mathcal{L}}\mathbb{N}$ has a non-zero
intersection with the imbedded copy of $C(X)$; a certain natural
representation of $C(X)\rtimes_{\alpha,\mathcal{L}}\mathbb{N}$ is
faithful.
It should be notices that since
$C(X)\rtimes_{\alpha,\mathcal{L}}\mathbb{N}$ can be constructed
as a \emph{singly generated dynamical system} (cf. \cite{DeaconuMR1233967}, \cite{Renault2MR1770333} and \cite{ExelVershikMR2195591}), it follows from Katsura's work
in \cite{KatsuraMR2253144} and \cite{KatsuraMR2279267} (which applies to a much more
general class of dynamical systems) that the first and the third
condition above are equivalent.
Theorem \ref{th:equivCovMaps} also extends a theorem in \cite{ExelVershikMR2195591} where it was shown that for
covering maps, topological freeness of the dynamical system implies
that every non-zero ideal in the crossed product
$C(X)\rtimes_{\alpha,\mathcal{L}}\mathbb{N}$ has a nonzero intersection
with the imbedded copy of $C(X)$ (we use this result in our proof of
Theorem \ref{th:equivCovMaps}). In the course of our investigation preceding
Theorem  \ref{th:equivCovMaps} we construct two representations
of $C(X)\rtimes_{\alpha,\mathcal{L}}\mathbb{N}$,
one having zero intersection of the kernel with the imbedded copy of $C(X)$
and the other being faithful. The condition of faithfulness
of the first representation is then shown to be one of the four equivalent
conditions in Theorem~\ref{th:equivCovMaps}.

\section{Crossed product $C^*$-algebra for
dynamics of covering maps} Let $X$ be a compact
Hausdorff space and let $T:X\to X$ be a covering
map, i.e., $T$ is continuous and surjective and
there exists for every $x\in X$ an open
neighborhood $V$ of $x$ such that $T\inv(V)$ is a
disjoint union of open sets
$(U_\alpha)_{\alpha\in I}$ satisfying that $T$
restricted to each $U_\alpha$ is a homeomorphism
from $U_\alpha$ onto $V$.

If we let $\alpha$, $\pret$ and $\trans$ be the
maps from $C(X)$ to
$C(X)$ given by
\begin{equation*}
  \alpha(f)= f\circ T,
\end{equation*}
\begin{equation*}
  \pret(f)(x)=\sum_{y\in T\inv(x)}f(y),
\end{equation*}
and
\begin{equation*}
  \trans(f)=\pret(1_X)\inv\pret(f),
\end{equation*}
(that these maps are well defined and maps $C(X)$
into $C(X)$ is showed in
\cite{ExelVershikMR2195591}), then $\trans$ is a
transfer operator for $\alpha$.

We will as in \cite{ExelVershikMR2195591} denote
$\alpha(\pret(1_X))$ by $\ind(E)$, and we will
for every $k\ge 1$ let
\begin{equation*}
  I_k = \ind(E)\alpha(\ind(E))\dotsm \alpha^{k-1}(\ind(E)).
\end{equation*}

% \begin{lemma}
%   Let $n\ge 1$, $f\in C(X)$ and $x\in X$. Then we have
%   \begin{equation*}
%     \trans^n(f)(x)=\sum_{y\in T^{-n}(x)}I_n\inv(y) f(y).
%   \end{equation*}
% \end{lemma}

% \begin{proof}
%   The equation holds by definition of $\trans$ for $n=1$. Assume now
%   that $k>1$ and that the equation holds for $n=k-1$. Then we have
%   \begin{equation*}
%     \begin{split}
%       \trans^k(f)(x)&=\sum_{y\in T^{-(k-1)}(x)}I_{k-1}\inv(y)
%       \trans(f)(y)\\
%       &=\sum_{y\in T^{-(k-1)}(x)}I_{k-1}\inv(y)\sum_{z\in
%         T\inv(y)}I_1\inv(z)f(z)\\
%       &=\sum_{z\in T^{-k}(x)}I_{k-1}\inv(T(z))I_1\inv(z)f(z)\\
%       &=\sum_{z\in T^{-k}(x)}I_k\inv(z)f(z),\\
%     \end{split}
%   \end{equation*}
%   which proves that the equation holds for all $n\ge 1$.
% \end{proof}

Since $\trans$ is a transfer operator for
$\alpha$, one can associate the $\cs$-algebra
$\cros$ to the dynamical system $(X,T)$ (here
$\cros$ is the crossed-product associated to the
triple $(C(X),\alpha,\trans)$ by Exel in
\cite{ExelMR2032486}). Exel and Vershik have in
\cite{ExelVershikMR2195591} studied this
$\cs$-algebra. Since $T$ is a covering map there
exists a finite open covering $\{V_i\}_{i=1}^t$
of $X$ such that the restriction of $T$ to each
$V_i$ is injective. Let $\{v_i\}_{i=1}^t$ be a
partition of unit subordinate to
$\{V_i\}_{i=1}^t$ and let
\begin{equation*}
u_i=(\alpha(\pret(1_X))v_i)^{1/2}.
\end{equation*}
Exel and Vershik have characterized $\cros$ by the following theorem:
\begin{theorem}[{\cite[Theorem 9.2]{ExelVershikMR2195591}}] \label{theorem:universal}
  The $\cs$-algebra $\cros$ is the universal $\cs$-algebra generated
  by a copy of $C(X)$ and an isometry $s$ subject to the relations
  \begin{enumerate}[\indent (1)]
  \item $sf=\alpha(f)s$,
  \item $s^*fs=\trans(f)$,
  \item $1=\sum_{i=1}^tu_iss^*u_i$,
  \end{enumerate}
  for all $f\in C(X)$.
\end{theorem}

\section{Two representations of $\cros$}
\label{sec:tworeps}

Let $X$ be a compact Hausdorff space and let $T:X\to X$ be a
covering map. Let $\hilbert$ be a Hilbert space with an orthonormal
basis $(e_x)_{x\in X}$ indexed by $X$. For $f\in C(X)$, let $M_f$ be
the bounded operator on $\hilbert$ defined by
\begin{equation*}
  M_f(e_x)=f(x)e_x,\quad x\in X,
\end{equation*}
and let $S$ be the bounded operator on $\hilbert$ defined by
\begin{equation*}
  S(e_x)=(\pret(1_X)(x))^{-1/2}\sum_{y\in T\inv(\{x\})}e_y,\quad x\in X.
\end{equation*}
It is easy to check that we for all $k\ge 1$ and all $x\in X$ have
\begin{equation*}
  S^k(e_x)=\sum_{y\in (T^k)\inv(\{x\})}(I_k(y))^{-1/2}e_y,
\end{equation*}
and
\begin{equation*}
  (S^*)^k(e_x)=(I_k(x))^{-1/2}e_{T^k(x)}.
\end{equation*}

\begin{proposition} \label{prop:rep}
  Let $X$ be a compact Hausdorff space, let $T:X\to X$ be a
  covering map and let $\hilbert$, $M_f$ and $S$ be as above. Then
  there exists a representation $\psi$ of $\cros$ on $\hilbert$ such
  that $\psi(f)=M_f$ for every $f\in C(X)$ and $\psi(s)=S$.

  We furthermore have that $\ker(\psi)\cap C(X)=\{0\}$.
\end{proposition}

\begin{proof}
  Let $f\in C(X)$. It is straight forward to check that
  $SM_f=M_{\alpha(f)}S$, that $S^*M_fS=M_{\trans(f)}$, and that
  $\sum_{i=1}^tM_{u_i}SS^*M_{u_i}=M_{1_X}$, so the existence of $\psi$
  follows from Theorem \ref{theorem:universal}.

  Let $f\in C(X)$, and assume that there is an $x\in X$ such that
  $f(x)\ne 0$. Then $M_fe_x\ne 0$, so $f\notin\ker(\psi)$, which
  proves that $\ker(\psi)\cap C(X)=\{0\}$.
\end{proof}
Let $\widetilde{\hilbert}$ be a Hilbert space
with an orthonormal basis $(e_{(x,n)})_{(x,n)\in
X\times\Z}$ indexed by $X\times\Z$. For $f\in
C(X)$, let $\widetilde{M}_f$ be the bounded
operator on $\widetilde{\hilbert}$ defined by
\begin{equation*}
  \widetilde{M}_f(e_{(x,n)})=f(x)e_{(x,n)},\quad (x,n)\in X\times\Z,
\end{equation*}
and let $\widetilde{S}$ be the bounded operator on
$\widetilde{\hilbert}$ defined by
\begin{equation*}
  \widetilde{S}(e_{(x,n)})=(\pret(1_X)(x))^{-1/2}\sum_{y\in
    T\inv(\{x\})}e_{(y,n+1)},\quad (x,n)\in X\times\Z.
\end{equation*}
It is easy to check that we for all $k\ge 1$, all $n\in\Z$ and all $x\in X$ have
\begin{equation*}
  \widetilde{S}^k(e_{(x,n)})=\sum_{y\in (T^k)\inv(\{x\})}(I_k(y))^{-1/2}e_{(y,n+k)},
\end{equation*}
and
\begin{equation*}
  (\widetilde{S}^*)^k(e_{(x,n)})=(I_k(x))^{-1/2}e_{(T^k(x),n-k)}.
\end{equation*}

\begin{proposition} \label{prop:faithful rep}
  Let $X$ be a compact Hausdorff space, let $T:X\to X$ be a
  covering map and let $\widetilde{\hilbert}$, $\widetilde{M}_f$ and
  $\widetilde{S}$ be as above. Then
  there exists a faithful representation $\widetilde{\psi}$ of $\cros$ on
  $\widetilde{\hilbert}$ such
  that $\widetilde{\psi}(f)=\widetilde{M}_f$ for every $f\in C(X)$ and
  $\widetilde{\psi}(s)=\widetilde{S}$.
\end{proposition}

\begin{proof}
  Let $f\in C(X)$. It is straight forward to check that
  $\widetilde{S}\widetilde{M}_f=\widetilde{M}_{\alpha(f)}S$, that
  $\widetilde{S}^*\widetilde{M}_f\widetilde{S}=\widetilde{M}_{\trans(f)}$,
  and that $\sum_{i=1}^t \widetilde{M}_{u_i} \widetilde{S}
  \widetilde{S}^* \widetilde{M}_{u_i} = \widetilde{M}_{1_X}$, so the
  existence of $\widetilde{\psi}$ follows from Theorem \ref{theorem:universal}.

  Let $z\in\T$. Let $U_z$ be the bounded operator on
  $\widetilde{\hilbert}$ defined by
  \begin{equation*}
    U_ze_{(x,n)}=z^ne_{(x,n)},\quad (x,n)\in X\times\Z.
  \end{equation*}
  We then have that $U_z\widetilde{M}_fU_z^*=\widetilde{M}_f$ for
  every $f\in C(X)$, and that
  $U_z\widetilde{S}U_z^*=z\widetilde{S}$. Thus
  \begin{equation*}
    (D,z)\mapsto U_zDU_z^*,\quad D\in B(\widetilde{H}),\ z\in\T
  \end{equation*}
  is an action of the circle group on $B(\widetilde{\hilbert})$ such
  that $\widetilde{\psi}$ is covariant relative to this action and the
  gauge action on $\cros$.

  Let $f\in C(X)$, and assume that there is an $x\in X$ such that
  $f(x)\ne 0$. Then $\widetilde{M}_fe_{(x,0)}\ne 0$, so
  $f\notin\ker(\widetilde{\psi})$, which
  proves that $\widetilde{\psi}$ is faithful on $C(X)$. Thus it follows
  from \cite[Theorem 4.2]{ExelVershikMR2195591}
  that $\widetilde{\psi}$ is faithful.
\end{proof}

Recall from \cite[Theorem
8.9]{ExelVershikMR2195591} that there exists a
conditional expectation $G:\cros\to C(X)$ such
that $G(fs^k(s^*)^lg)=\delta_{k,l}fI\inv_ng$ for
$f,g\in C(X)$ and $k,l\in\N$, where $\delta$ is
the Kronecker symbol.

\begin{lemma} \label{lemma:b is cont}
  Let $X$ be a compact Hausdorff space, let $T:X\to X$ be a
  covering map and let $\widetilde{\psi}$ be the representation of Proposition
  \ref{prop:faithful rep}. Then we have for all $b\in\cros$:
  \begin{enumerate}[\indent (1)]
  \item $\inner{\widetilde{\psi}(b)e_{(x,n)},e_{(x,n)}}=
    G(b)(x)$ for all $(x,n)\in X\times\Z$.
  \item $b\in C(X)$ if and only if $\inner{
    \widetilde{\psi}(b)e_{(x,m)},e_{(y,n)}}=0$ for all
    $(x,m),(y,n)\in X\times\Z$ with $(x,m)\ne (y,n)$.
  \item If $(x,m),(y,n)\in X\times\Z$ and $\inner{
    \widetilde{\psi}(b)e_{(x,m)},e_{(y,n)}}\ne 0$, then there
    exist $k,l\in\N$ such that $k+m=l+n$ and such that
    $T^l(x)=T^k(y)$, and there exist an open neighbourhood $U$ of $x$
    and an open neighbourhood $V$ of $y$ such that if $x'\in U$,
    $y'\in V$ and $T^l(x')=T^k(y')$, then $\inner{
    \widetilde{\psi}(b)e_{(x',m)},e_{(y',n)}}\ne 0$.
  \end{enumerate}
\end{lemma}

\begin{proof}
  (1): Fix $(x,n)\in X\times\Z$. It follows from
  \cite[Proposition 2.3]{ExelVershikMR2195591} that the set
  of elements of $\cros$ which can be written as a finite sum of
  elements of the form $fs^k(s^*)^lg$ where $f,g\in C(X)$ and
  $k,l\in\N$, is dense in $\cros$. Since both $b\mapsto\inner{
  \widetilde{\psi}(b)e_{(x,n)},e_{(x,n)}}$ and $b\mapsto
  G(b)(x)$ are linear and continuous maps, it follows that it is
  enough to show that if $f,g\in C(X)$ and
  $k,l\in\N$, then $\inner{
  \widetilde{\psi}(fs^k(s^*)^lg)e_{(x,n)},e_{(x,n)}}=
  \delta_{k,l}f(x)I\inv_k(x)g(x)$, and it is straight forward to check
  that this is the case.

  (2): It is clear that if $f\in C(X)$, then
  $\inner{\widetilde{\psi}(f)e_{(x,m)},e_{(y,n)}}=0$ for all
  $(x,m),(y,n)\in X\times\Z$ with $(x,m)\ne (y,n)$.

  Let $b\in\cros$ and assume that $\inner{
  \widetilde{\psi}(b)e_{(x,m)},e_{(y,n)}}=0$ for all
  $(x,m),(y,n)\in X\times\Z$ with $(x,m)\ne (y,n)$. It then follows
  from what we have just shown that $\inner{
  \widetilde{\psi}(b)e_{(x,m)},e_{(y,n)}}=\inner{
  \widetilde{\psi}(G(b))e_{(x,m)},e_{(y,n)}}$ for all
  $(x,m),(y,n)\in X\times\Z$, and thus that $\widetilde{\psi}(b)=
  \widetilde{\psi}(G(b))$. Since $\widetilde{\psi}$ is faithful, it
  follows that $b=G(b)\in C(X)$.

  (3): Let $(x,m),(y,n)\in X\times\Z$ and assume that $\varepsilon:=
  \inner{\widetilde{\psi}(b)e_{(x,m)},e_{(y,n)}}\ne 0$. It
  follows from \cite[Proposition 2.3]{ExelVershikMR2195591} that there exists a
  finite subset $F$ of $C(X)\times\N\times\N\times C(X)$ such that
  \begin{equation*}
      \norm[\Big]{b-\sum_{(f,k,l,g)\in F}fs^k(s^*)^lg}< \varepsilon/3.
    \end{equation*}
    If $(f,k,l,g)\in F$, then we have
    \begin{equation*}
      \widetilde{M}_f\widetilde{S}^k (\widetilde{S}^*)^l\widetilde{M}_g
      e_{(x,m)}= \sum_{y'\in (T^k)\inv(T^l(\{x\}))} f(y')
      (I_k(y')I_l(y'))^{-1/2} g(x)e_{(y',m-l+k)},
    \end{equation*}
    so either do we have that $y\in (T^k)\inv(T^l(\{x\}))$ and
    $n=m-l+k$, and thus
    that $k+m=l+n$ and $T^l(x)=T^k(y)$, or we have $\inner{
    \widetilde{\psi}(fs^k(s^*)^lg)e_{(x,m)},e_{(y,n)}}=0$.

    Let $F':=\{(f,k,l,g)\in F\mid k+m=l+n\text{ and
    }T^l(x)=T^k(y)\}$. Since $\sum_{(f,k,l,g)\in F}\inner{
    \widetilde{\psi}(fs^k(s^*)^lg)e_{(x,m)},e_{(y,n)}}\ne 0$, it
    follows that $F'\ne\emptyset$. Let $r$ be
    the number of elements of $F'$, let
    $k:=\max\{k'\mid (f',k',l',g')\in F'\}$ and let $l:=k+m-n$. Choose
    for each $(f',k',l',g')\in F'$ an open neighbourhood
    $U_{(f',k',l',g')}$ of $x$ and an open neighbourhood
    $V_{(f',k',l',g')}$ of $y$ such that there for each $x'\in
    U_{(f',k',l',g')}$ exists a unique $y'\in V_{(f',k',l',g')}$ such
    that $T^{l'}(x')=T^{k'}(y')$, and such that
    \begin{equation*}
      \abs{f(y)(I_k(y)I_l(y))^{-1/2} g(x)-f(y')
      (I_k(y')I_l(y'))^{-1/2} g(x')}<\varepsilon/(3r)
    \end{equation*}
    for $x'\in U_{(f',k',l',g')}$ and $y'\in V_{(f',k',l',g')}$ with
    $T^{l'}(x')=T^{k'}(y')$.

    Let $U=\bigcap_{(f',k',l',g')\in F'}U_{(f',k',l',g')}$ and
    $V=\bigcap_{(f',k',l',g')\in F'}V_{(f',k',l',g')}$. Then $U$ is an
    open neighbourhood of $x$, $V$ is an open neighbourhood of $y$,
    and if $x'\in U$, $y'\in V$ and $T^l(x')=T^k(y')$, then we have
    \begin{alignat*}{2}
      \abs{\inner{\widetilde{\psi}(b)e_{(x',m)},e_{(y',n)}}} &\ge&
      \varepsilon &-
      \abs{\inner{\widetilde{\psi}(b)e_{(x,m)},e_{(y,n)}} -
        \inner{\widetilde{\psi}(b)e_{(x',m)},e_{(y',n)}}}
        \end{alignat*}
    \begin{alignat*}{2}
      &\ge&
      \varepsilon &-
      \abs[\Big]{\inner{\widetilde{\psi}(b)e_{(x,m)},e_{(y,n)}} -
        \inner[\Big]{\sum_{(f,k,l,g)\in F}\widetilde{M}_f \widetilde{S}^k
          (\widetilde{S}^*)^l
          \widetilde{M}_ge_{(x,m)},e_{(y,n)}}} \\
      &&&- \abs[\Big]{\inner[\Big]{\sum_{(f,k,l,g)\in
            F}\widetilde{M}_f \widetilde{S}^k
          (\widetilde{S}^*)^l \widetilde{M}_ge_{(x,m)},e_{(y,n)}}\\
        &&&\qquad - \inner[\Big]{\sum_{(f,k,l,g)\in F}\widetilde{M}_f
          \widetilde{S}^k (\widetilde{S}^*)^l
          \widetilde{M}_ge_{(x',m)},e_{(y',n)}}}\\
      &&&- \Bigl\lvert
        \inner[\Big]{ \sum_{(f,k,l,g)\in F}
          \widetilde{M}_f \widetilde{S}^k
          (\widetilde{S}^*)^l \widetilde{M}_ge_{(x',m)},e_{(y',n)}}
       - \inner{\widetilde{\psi}(b)e_{(x',m)},e_{(y',n)}}
      \Bigr\rvert \\
      & >&
      \varepsilon &- \varepsilon/3
      - \sum_{(f,k,l,g)\in F'}\bigl\lvert f(y)(I_k(y)I_l(y))^{-1/2}g(x)\\
        &&&\qquad-f(y')(I_k(y')I_l(y'))^{-1/2}g(x')\bigr\rvert -\varepsilon/3
      >0.
        \end{alignat*}
\end{proof} %

\section{The main theorem}

Let $X$ be a compact Hausdorff space and let
$T:X\to X$ be a covering map. As in
\cite{ExelVershikMR2195591}, we say that $(X,T)$
is \emph{topological free} if for every pair of
nonnegative integers $(k,l)$ with $k\ne l$, the
set $\{x\in X\mid T^k(x)=T^l(x)\}$ has empty
interior. If the space $X$ is infinite, and we consider
dynamical systems generated by covering maps, then the
class of topologically free systems contains the
subclass of irreducible dynamical systems,
defined as follows (see \cite[Proposition 11.1]{ExelVershikMR2195591}). Two points $x,y\in X$ are
said to be {\em trajectory-equivalent} (see
e.g.~\cite{ArumzVershik1}) when there are
$n,m\in\N$ such that $T^n(x) = T^m(y)$.  We will
denote this by $x\sim y$. A subset $Y\subseteq X$
is said to be invariant if $x\sim y \in Y$
implies that $x\in Y$.  It is easy to see that
$Y$ is invariant if and only if\/ $T\inv(Y)=Y$.
The covering map $T$ and the dynamical system it
generates  is said to be {\em irreducible} when
there is no closed (equivalently open) invariant
set other than $\emptyset$ and $X$ (see
e.g.~\cite{ArumzVershik1}). Notice that
irreducibility is weaker than the condition of
minimality defined in \cite{DeaconuMR1233967}. In
\cite{ExelVershikMR2195591}, it was shown that,
for dynamical systems generated by covering maps
of infinite spaces, irreducibility of the system
is equivalent to simplicity of $\cros$.

\begin{theorem} \label{th:equivCovMaps}
  Let $X$ be a compact Hausdorff space, and let $T:X\to X$ be a
  covering map. Then the following are equivalent:
  \begin{enumerate}[\indent (1)]
  \item $(X,T)$ is topological free.
  \item Every nontrivial ideal of $\cros$ has a nontrivial
    intersection with $C(X)$.
  \item The representation of Proposition \ref{prop:rep} is faithful.
  \item $C(X)$ is a maximal abelian $C^*$-subalgebra of $\cros$.
  \end{enumerate}
\end{theorem}

\begin{proof}
  (1)$\implies$(2): That (1) implies (2) is proven in \cite[Theorem
  10.3]{ExelVershikMR2195591}.

  (2)$\implies$(3): Assume that (2) holds and let $\psi$ be the
  representation of Proposition \ref{prop:rep}. It follows from Proposition
  \ref{prop:rep} that $\ker(\psi)\cap C(X)=\{0\}$, so
  $\ker(\psi)=\{0\}$. Thus (3) holds.

  (1)$\implies$(4): Assume that $(X,T)$ is topological free. Let
  $b\in\cros$ and assume that $bf=fb$ for all $f\in C(X)$. Let
  $\widetilde{\psi}$ be the representation of Proposition
  \ref{prop:faithful rep}. We want to show that $b\in C(X)$. It
  follows from Lemma \ref{lemma:b is cont}(2) that it is enough to show that
  $\inner{\widetilde{\psi}(b)e_{(x,m)}, e_{(y,n)}}=0$ for all
  $(x,m),(y,n)\in X\times\Z$ with $(x,m)\ne (y,n)$.

  Fix $(x,m),(y,n)\in X\times\Z$ with $(x,m)\ne (y,n)$. Assume first
  that $x\ne y$. Choose
  $f\in C(X)$ such that $f(x)\ne 0$ and $f(y)=0$. We then have that
  \begin{equation*}
    \begin{split}
      f(x)\inner{\widetilde{\psi}(b)e_{(x,m)}, e_{(y,n)}}&=
      \inner{\widetilde{\psi}(bf)e_{(x,m)}, e_{(y,n)}}=
      \inner{\widetilde{\psi}(fb)e_{(x,m)}, e_{(y,n)}}\\
      &= \inner{\widetilde{\psi}(b)e_{(x,m)},
      \widetilde{M}_{\overline{f}}e_{(y,n)}} =0,
    \end{split}
  \end{equation*}
  so $\inner{\widetilde{\psi}(b)e_{(x,m)}, e_{(y,n)}}=0$ as wanted.

  Assume then that $x=y$ and $m\ne n$. Assume that
  $\inner{\widetilde{\psi}(b)e_{(x,m)}, e_{(y,n)}}\ne 0$.
  It follows from Lemma \ref{lemma:b is cont}(3) that there exist
  $k,l\in\N$ such that $T^l(x)=T^k(x)$ and $k+m=l+n$, and two open
  neighbourhoods $U$ and $V$ of $x$
  such that if $x'\in U$, $y'\in V$ and $T^l(x')=T^k(y')$, then $\inner{
  \widetilde{\psi}(b)e_{(x',m)},e_{(y',n)}}\ne 0$. We must have
  that $k\ne l$, and since we are
  assuming that $(X,T)$ is topological free, it follows that there
  exist $x'\in U$ and $y'\in V$ such that $T^l(x')=T^k(y')$ and $x'\ne
  y'$. But it then follows from what we have just proved that $\inner{
  \widetilde{\psi}(b)e_{(x',m)},e_{(y',n)}}$ is both zero and
  non-zero, so we have a contradiction. Thus
  $\inner{\widetilde{\psi}(b)e_{(x,m)}, e_{(y,n)}}=0$ for all
  $(x,m),(y,n)\in X\times\Z$ with $(x,m)\ne (y,n)$. Hence $b\in C(X)$,
  which proves that (4) holds.

  (4)$\implies$(1) and (3)$\implies$(1): Assume that $(X,T)$ is not
  topological free. We will then show that the representation of Proposition \ref{prop:rep} is not
faithful, and that $C(X)$ is not maximal abelian. Since $(X,T)$ is not
  topological free, there exist $k\ne l$ such that $\{x\in X\mid T^k(x)=T^l(x)\}$
  has non-empty interior. Choose a non-empty open subset $U$ of
  $\{x\in X\mid T^k(x)=T^l(x)\}$ such that $T^k$ and $T^l$ are
  injective on $U$, and let $x_0\in U$. Choose an $f\in C(X)$ with $\supp
  f\subseteq U$ and $f(x_0)\ne 0$.

  Let $\{V_i\}_{i=1}^t$ be a finite open covering of $X$ such that the
  restriction of $T$ to each $V_i$ is injective, let $\{v_i\}_{i=1}^t$
  be a partition of unit subordinate to $\{V_i\}_{i=1}^t$, let
  $Z:=\{1,2,\dots,t\}$,  let $u_i:=(\alpha(\pret(1_X))v_i)^{1/2}$ for
  $i\in Z$, and let
  \begin{equation*}
    u_j:=u_{j_0}\alpha(u_{j_1})\alpha^2(u_{j_2})\dots\alpha^{l-1}(u_{j_{l-1}})
  \end{equation*}
  for $j=(j_0,j_1,j_2\dots,j_{l-1})\in Z^l$. It then follows from
  Proposition 7.4 and Proposition 8.2 of
  \cite{ExelVershikMR2195591} that $\sum_{j\in
    Z^l}u_js^l(s^*)^lu_j^*=1$.

  Let $h\in C(X)$. If $j\in Z^l$, then we have
  $f\alpha^k(\trans^l(hfu_j))=hf\alpha^k(\trans^l(fu_j))$. It follows
  that we have
  \begin{equation*}
    \begin{split}
      fs^k(s^*)^lfh&=fs^k(s^*)^lhf=fs^k(s^*)^lfh\sum_{j\in Z^l}u_js^l(s^*)^lu_j^*\\
      &=fs^k\sum_{j\in Z^l}\trans^l(fhu_j)(s^*)^lu_j^*=\sum_{j\in Z^l}f\alpha^k(\trans^l(hfu_j))s^k(s^*)^lu_j^*\\
      &=\sum_{j\in Z^l}hf\alpha^k(\trans^l(fu_j))s^k(s^*)^lu_j^*=hfs^k(s^*)^lf\sum_{j\in Z^l}u_js^l(s^*)^lu_j^*\\
      &=hfs^k(s^*)^lf.
    \end{split}
  \end{equation*}
  Thus $fs^k(s^*)^lf\in C(X)'$.

  Let $\widetilde{\psi}$ be the representation of Proposition
  \ref{prop:faithful rep}. It is easy to check that
  \begin{equation*}
    \inner{\widetilde{\psi}(fs^k(s^*)^lf)e_{(x_0,l)},
  e_{(x_0,k)}}=f(x_0)(I_k(x_0)I_l(x_0))^{-1/2}f(x_0)\ne 0,
  \end{equation*}
  and that $\inner{\widetilde{\psi}(h)e_{(x_0,l)},
  e_{(x_0,k)}}=0$ for all $h\in C(X)$. It follows that
  $fs^k(s^*)^lf\notin C(X)$.
%   Since $\trans^n(f)\trans^m(f)(T^n(x))\ne 0$, we also have that
%   \begin{equation*}
%     (s^*)^n(fs^n(s^*)^mf)s^m=\trans^n(f)\trans^m(f)\ne 0
%   \end{equation*}
%   which means that $fs^n(s^*)^mf$ does not belong to $\mathcal{U}$ and
%   thus not $C(X)$ because $F(s^nb(s^*)^m)=0$ for every $b\in
%   \mathcal{U}$.
  This shows that $C(X)'\setminus C(X)$ is non-empty,
  and thus that $C(X)$ is not a maximal abelian $C^*$-subalgebra of
  $\cros$. Hence $\neg (1)\implies\neg (4)$ or
  equivalently $(4)\implies (1)$.

  Let $\psi$ be the representation of Proposition \ref{prop:rep}. It
  is easy to check that we for all $x,y\in X$ have
  $\inner{\psi(fs^k(s^*)^lf)e_x,e_y}=
  \delta_{x,y}f(x)(I_k(x)I_l(x))^{-1/2}f(x)$, and thus that
  $\psi(fs^k(s^*)^lf)=\psi(f(I_kI_l)^{-1/2}f)$. We have already seen
  that $fs^k(s^*)^lf\notin C(X)$, and since $f(I_kI_l)^{-1/2}f\in
  C(X)$, it follows that $\psi$ is not faithful. Hence $\neg
  (1)\implies\neg (3)$ or equivalently $(3) \implies (1)$.
\end{proof}


\begin{thebibliography}{11}

\bibitem{AdjiLacaNielsenRaeburn}
Adji, S., Laca, M., Nilsen, M.,
 and I. Raeburn,
\newblock Crossed products by semigroups of endomorphisms
  and the Toeplitz algebras of ordered groups,
  Proc. Amer. Math. Soc.,
  {\bf 122}, 1133--1141 (1994)

\bibitem{ArchbordSpielberg}
Archbold, R. J.,  Spielberg, J. S.,
\newblock Topologically free actions and ideals
in discrete $C^{*}$-dynamical systems.
\newblock Proc. of Edinburgh Math. Soc.
{\bf 37}, 119-124 (1993)

\bibitem{ArumzVershik1}
V. A. Arzumanian and A. M. Vershik,
\newblock Factor
representations of the crossed product of a
commutative C*-algebra and a semigroup of its
endomorphisms,
\newblock Dokl. Akad. Nauk. SSSR, {\bf 238},
513--517 (1978). Translated in Soviet
Math.~Dokl., {\bf 19} , No.~1 (1978)

\bibitem{ArumzVershik2}
Arzumanian, V. A., Vershik, A. M.,
\newblock Star algebras associated with endomorphisms,
\newblock Operator Algebras and
Group Representations,
\newblock Proc. Int. Conf., Neptun/Rom. 1980, Vol. I,
Monogr. Stud. Math. 17, 17-27 (1984)

\bibitem{BratEvansJorg} Bratteli, O., Evans, D.
E., Jorgensen, P. E. T.,
\newblock Compactly supported
wavelets and representations of the Cuntz
relations.
\newblock Appl. Comput. Harmon. Anal. 8,
no. 2, 166--196 (2000).

\bibitem{BratJorgIFSAMSmemo}
Bratteli, O., Jorgensen, P.E.T., Iterated
function systems and permutation representations
of the Cuntz algebra. Memoirs Am. Math. Soc.
139(663), 1--89 (1999)

\bibitem{BratJorgbook} Bratteli, O.,
Jorgensen, P.,
\newblock Wavelets through a looking glass.
The world of the spectrum,
\newblock Applied and Numerical
Harmonic Analysis. Birkhauser Boston, Inc.,
Boston, MA, xxii+398 pp. 2002

\bibitem{CarlsenSilv} Carlsen, T. M., Silvestrov, S.,
\newblock $C^*$-crossed products and shift spaces,
\newblock Expo. Math. 25, no. 4, 275--307 (2007)

\bibitem{CuntzKr1} Cuntz, J., Krieger, W.,
\newblock A class of
$C\sp{*}$-algebras and topological Markov chains,
\newblock Invent. Math. 56, no. 3, 251--268 (1980)

\bibitem{DaiLarsonMeAMS} Dai X., Larson D. R.,
\newblock Wandering vectors for unitary systems and
orthogonal wavelets,
\newblock Mem. Amer. Math. Soc. 134, no. 640,
68 pp., 1998

\bibitem{Davidson}  Davidson, K. R.,
\newblock $C^*$-algebras by example,
\newblock Fields Institute Monographs, AMS, 1996

\bibitem{DeaconuMR1233967} Deaconu, V.,
\newblock Groupoids associated with endomorphisms,
\newblock Trans. Amer. Math. Soc.,
{\bf 347}, 5, 1779--1786 (1995)

\bibitem{DutkayJorg1}
Dutkay, D., Jorgensen, P. E. T.,
\newblock Wavelet
constructions in non-linear dynamics,
\newblock Electron.
Res. Announc. Amer. Math. Soc. 11, 21--33 (2005)

\bibitem{DutkayJorg2} Dutkay D., Jorgensen P. E. T.,
\newblock Methods from
multiscale theory and wavelets applied to
nonlinear dynamics,
\newblock Wavelets, multiscale systems
and hypercomplex analysis, Oper. Theory Adv.
Appl., 167, Birkhauser, Basel, 87--126 (2006)

\bibitem{DutkayJorg3} Dutkay, D. E.,
Jorgensen, P. E. T.,
\newblock Martingales, endomorphisms,
and covariant systems of operators in Hilbert
space,
\newblock J. Operator Theory {\bf 58}, no. 2,
269--310 (2007)

\bibitem{DutkayJorg4} Dutkay, D. E.,
Jorgensen, P. E. T.,
\newblock Hilbert spaces of
martingales supporting certain
substitution-dynamical systems,
\newblock Conform. Geom. Dyn. 9, 24--45 (2005)

\bibitem{EffrosHahn} Effros, E. G., Hahn, F.,
\newblock Locally compact transformation groups and
\mbox{$C^{*}$-algebras},
\newblock Mem. Amer. Math.
Soc. {\bf 75}, 1-92 (1967)

\bibitem{Eilers1} Eilers, S.,
\newblock $C^*$-algebras
associated to dynamical systems,
\newblock Discrete Contin.
Dyn. Syst. 15, no. 1, 177--192 (2006).

\bibitem{Elliott}  Elliott, G. A.,
\newblock Some simple \mbox{$C^{*}$-algebras}
constructed as crossed products with discrete
outer automorphism groups,
\newblock Publ. Res. Inst.
Math. Sci. {\bf 16}, 299-311 (1980)

\bibitem{ExelMR1966826}
Exel, R.,
\newblock Crossed-products by finite
index endomorphisms and {KMS} states.
\newblock J. Funct. Anal.,
{\bf 199} (1), 153--188 (2003)

\bibitem{ExelMR2032486}
 Exel, R.,
\newblock A new look at the crossed-product
of a {$C^*$}-algebra by an endomorphism.
\newblock Ergodic Theory Dynam. Systems,
{\bf 23} (6), 1733--1750 (2003)

\bibitem{ExelVershikMR2195591}
Exel, R., Vershik, A.,
\newblock $C^*$-algebras of irreversible dynamical systems.
\newblock Canad. J. Math.
{\bf 58} (1), 39--63 (2006)

\bibitem{anHuefRaeburnMR98k:46098} an Huef, A.,
Raeburn, I., The ideal structure of Cuntz-Krieger
algebras, Ergodic Theory Dynam. Systems, {\bf 17}
(3), 611 -- 624 (1997)

\bibitem{Jorg-b-1} Jorgensen, P. E. T.,
\newblock Operators and Representation Theory,
\newblock North-Holland, \mbox{Amsterdam,} 1988.

\bibitem{JorgWavSignFracbook}
Jorgensen, P. E. T.,
\newblock Analysis and probability:
wavelets, signals, fractals,
\newblock Graduate Texts in Mathematics, 234,
Springer, New York, xlviii+276 pp, 2006

\bibitem{KajiwaraWatataniComlDyn}
Kajiwara, T., Watatani, Y.,
\newblock $C^*$-algebras
associated with complex dynamical systems,
\newblock Indiana Univ. Math. J. {\bf 54} ,
no. 3, 755--778 (2005)

\bibitem{KatsuraMR2253144}
Katsura, T.,
\newblock A class of {$C\sp *$}-algebras generalizing both graph algebras and
  homeomorphism {$C\sp *$}-algebras. {II}. Examples.
\newblock Internat. J. Math., {\bf 17}, 7, 791--833 (2006).

\bibitem{KatsuraMR2279267}
Katsura, T.,
\newblock A class of {$C\sp *$}-algebras generalizing both graph algebras and
  homeomorphism {$C\sp *$}-algebras. {III}. {I}deal structures.
\newblock {\em Ergodic Theory Dynam. Systems}, {\bf 26}, 6, 1805--1854 (2006).

\bibitem{Kawamura1} Kawamura, S.,
\newblock Covariant representations associated
with chaotic dynamical systems,
\newblock Tokyo J. Math. 20,
no. 1, 205--217 (1997).

\bibitem{KawamuraTomiama}
Kawamura, S., Tomiyama, J.,
\newblock Properties of topological dynamical
systems and corresponding $C^*$-algebras,
\newblock Tokyo. J. Math.
{\bf 13}, 251--257 (1990)

\bibitem{Kishimoto}  Kishimoto, A.,
\newblock Outer automorphisms
and reduced crossed products of simple
$C^*$-algebras,
\newblock Comm. Math. Phys.
{\bf 81}, 429--435 (1981)

\bibitem{MACbook1} Mackey, G. W.,
\newblock Induced Representations of Groups and
Quantum Mechanics,
\newblock W. A. Benjamin, New
York, Editore Boringhieri, Torino, 1968.

\bibitem{MACbook2} Mackey, G. W.,
\newblock The Theory of Unitary Group
Representations,
\newblock The University of Chicago
Press, 1976.

\bibitem{MACbook3} Mackey, G. W.,
\newblock Unitary Group Representations
in Physics, Probability, and Number Theory,
\newblock Addison-Wesley, 1989.

\bibitem{Matsumoto1} Matsumoto, K.,
\newblock On $C^*$-algebras associated with subshifts,
\newblock Internat. J. Math. {\bf 8}, no. 3,
357--374 (1997)

\bibitem{OstSam-book}
Ostrovsky\u{\i}, V. L.,  Samo\u{\i}lenko, Yu. S.,
\newblock Introduction to the Theory
of Representations of Finitely Presented
$*$-Algebras. I. Representations by bounded
operators,
\newblock Rev. Math. Phys. {\bf 11}, The
Gordon and Breach Publ. Group, 1999.

\bibitem{Ped-book}  Pedersen, G. K.,
\newblock
$C^*$-algebras and their automorphism groups,
\newblock Academic Press, 1979.

\bibitem{QuiggSpielberg}
Quigg, J. C.,  Spielberg, J. S.,
\newblock Regularity and hyporegularity in
{$C^*$}-dynamical system,
\newblock {\em Houston J. Math.} {\bf 18},
139-152 (1992)

\bibitem{Renault1}
Renault, J.,
\newblock A groupoid approach to C*-algebras,
\newblock Lecture Notes in Mathematics vol.~793,
Springer, 1980

\bibitem{Renault2MR1770333}
Renault, J.,
\newblock Cuntz-like algebras,
\newblock In {\em Operator theoretical methods (Timi\c soara, 1998)},
Theta Found., Bucharest, 371--386 (2000).

\bibitem{Renault3} Renault, J.,
\newblock Cartan subalgebras in $C^*$-algebras,
\newblock Irish Math. Soc. Bulletin {\bf 61}, 29-63 (2008)

\bibitem{STom}  Silvestrov, S. D.,  Tomiyama, J.,
\newblock Topological Dynamical Systems of Type I,
\newblock Expo. Math. \textbf{20}, 117-142 (2002)

\bibitem{Stacey}
\newblock Stacey, P. J.,
\newblock Crossed products of C*-algebras
by *-endomorphisms,
\newblock J. Aust. Math. Soc., Ser. A,
{\bf 54}, 204--212 (1993)

\bibitem{Spielberg}
Spielberg, J. S.,
\newblock Free-product groups, Cuntz-Krieger
algebras, and covariant maps,
\newblock {\em Internat. J. Math.} {\bf 2},
457-476 (1991)

\bibitem{SSD1}  Svensson, C.,  Silvestrov S., de Jeu, M.,
\newblock Dynamical Systems and Commutants in
Crossed Products,
\newblock Internat. J. Math. \textbf{18},
455-471 (2007)

\bibitem{SSD2}  Svensson, C.,  Silvestrov S., de Jeu, M.,
\newblock Connections between dynamical systems and crossed
products of Banach algebras by $\mathbb{Z}$,
\newblock to appear in the proceedings
of ``Operator Theory, Analysis and Mathematical
Physics'', OTAMP-2006, Lund, Sweden, June 15-22,
2006. (Preprints in Mathematical Sciences 2007:5,
LUTFMA-5081-2007; Leiden Mathematical Institute
report 2007-02; arXiv:math/0702118).

\bibitem{SSD3}  Svensson, C.,  Silvestrov S., de Jeu, M.,
\newblock Dynamical systems associated with crossed
products,
\newblock to appear in the proceedings of
``Operator Methods in Fractal Analysis, Wavelets
and Dynamical Systems'', BIRS, Banff, Canada,
December 2 - December 7, 2006. (Preprints in
Mathematical Sciences 2007:22, LUTFMA-5088-2007;
Leiden Mathematical Institute report 2007-30;
arXiv:0707.1881).

\bibitem{SvT}  Svensson, C.,  Tomiyama, J.,
\newblock  On the commutant of $C(X)$
in $C^*$-crossed products
by $\mathbb{Z}$ and their representations,
\newblock  arXiv:0807.2940

\bibitem{Tomiama:SeoulLN}  Tomiyama, J.,
\newblock The interplay between
topological dynamics and theory of
$C^*$-algebras,
\newblock Lecture Notes Series,
{\bf 2}, Global Anal. Research Center, Seoul,
1992

\bibitem{Watatani1} Watatani, Y.,
\newblock Index for C*-subalgebras,
\newblock Mem. Am. Math. Soc.,
{\bf 424}, 117 pp. 1990

\bibitem{Williamsbook} Williams, D. P.,
\newblock Crossed
products of $C^*$-algebras,
\newblock Mathematical Surveys
and Monographs, 134, American Mathematical
Society, Providence, RI,  xvi+528 pp, 2007

\bibitem{Zeller-Meier}
Zeller-Meier, G.,
\newblock Produits crois\'es d'une
$C^*$-alg\`ebre par un groupe d'automorphismes,
\newblock J. Math. pures et appl. \textbf{47},
101--239 (1968)
\end{thebibliography}
\end{document}